\documentclass[12pt]{article}
\usepackage{epsfig,color}

\usepackage{times}
\usepackage{amsmath,amsfonts,amstext,amssymb,amsbsy,amsopn,amsthm,eucal}
\usepackage{txfonts}
\usepackage{dsfont}
\usepackage{graphicx}   

\numberwithin{equation}{section}
  \theoremstyle{plain}
 \newtheorem{theorem}[equation]{Theorem}

 \newtheorem{lemma}[equation]{Lemma}
 \newtheorem{corollary}[equation]{Corollary}

 \theoremstyle{remark}

 \newtheorem{remark}[equation]{Remark}

\theoremstyle{definition}
 \newtheorem{definition}[equation]{Definition}

\newcommand{\rexp}{{\rm exp}}
\newcommand{\Vol}{{\rm Vol}}

\newcommand{\inj}{{\rm inj}}

\newcommand{\dR}{\mathds{R}}

\newcommand{\dZ}{\mathds{Z}}

\newcommand{\dist}{{\rm dist}}

\newcommand{\cB}{\mathcal{B}}
\newcommand{\cC}{\mathcal{C}}

\newcommand{\cN}{\mathcal{N}}

\newcommand{\cS}{\mathcal{S}}

\newcommand{\cW}{\mathcal{W}}

\begin{document}

\title{Quantitative Stratification and the\\Regularity of Harmonic Maps and Minimal Currents}

\author{Jeff Cheeger\thanks{The author was partially supported by NSF grant DMS1005552 } and Aaron Naber\thanks{The author was partially supported by NSF postdoctoral grant 0903137}}
\date{\today}
\maketitle

\begin{abstract}
We introduce techniques for turning estimates on the
infinitesimal behavior of solutions to nonlinear equations
(statements concerning tangent cones and blow ups)  into more effective control.
In the present paper, we focus on proving regularity theorems for stationary and minimizing harmonic maps and minimal currents. There are several aspects to our improvements of known estimates.
First, we replace known estimates on the Hausdorff dimension of singular sets
by estimates on their Minkowski $r$-content, or equivalently, on the volumes of
their $r$-tubular neighborhoods.
  Second, we give improved regularity  control
with respect to the number of derivatives bounded and/or on the norm in which the derivatives
are bounded. As an example of the former,
  our results for  minimizing harmonic maps $f:M^n\rightarrow N^m$ between riemannian manifolds  include {\it a priori} bounds in $W^{1,p}\cap W^{2,\frac{p}{2}}$ for all $p<3$. These are
the first such bounds involving second derivatives in general dimensions. Finally,
the quantity  we control is actually provides much stronger information
than follows from a bound on the $L_p$ norm of derivatives.
Namely, we obtain $L_p$ bounds for the inverse of the {\it regularity scale} $r_f(x):= \max\left\{r: \sup_{B_r(x)}r|\nabla f|+r^2|\nabla^2 f|\leq 1\right\}$.
Applications to minimal hypersufaces include {\it a priori} $L_p$ bounds for the second fundamental form $A$ for all $p<7$.  Previously known bounds   were for $p\leq 2+\epsilon(n)$.
Again, the full theorem is much stronger and yields $L_p$ bounds for the corresponding regularity scale $r_{|A|}(x):= \max\left\{r: \sup_{B_r(x)}r|A|\leq 1\right\}$.  In outline, our discussion follows that of an earlier paper in which we proved analogous estimates in the context of noncollapsed riemannian manifolds with a lower bound
on Ricci curvature. These were applied to Einstein manifolds.  A key point in
all of these arguments is to establish the relevant quantitative differentiation theorem.
\end{abstract}
\pagebreak

\small{\tableofcontents}

\section{Introduction}
In this paper, we  study harmonic maps between smooth riemannian manifolds, and by similar methods, minimal and stationary currents on riemannian manifolds.  We introduce techniques which, when combined with ineffective tangent cone estimates, yield new effective regularity control.

Throughout the paper, $0^n\in\dR^n$ denotes the origin in $\dR^n$ and $x\in M^n$ denotes a point of the riemannian manifold $(M^n,g)$. Without essential loss of generality,
the following assumptions  will  be in force throughout the remainder of
the paper.
\begin{align}\label{con:sectional}
 |\sec_{B_2(x)}|\leq 1\, ,
\end{align}
\begin{align}
\label{con:volume}
 \inj_{M^n}(x)\geq 1>0\, .
\end{align}


Our main theorems are the quantitative stratifications of Theorems \ref{t:harm_quant_strat}, \ref{t:current_quant_strat} and the new regularity results of Theorems \ref{t:harm_min_regularity}, \ref{t:current_min_regularity} and Corollaries \ref{c:harm_min_regularity}, \ref{c:current_min_regularity}.

As an example, according to Corollary \ref{c:harm_min_regularity},  a minimizing harmonic map
$f:(M,g)\to (N^m,h)$ has a proiri bounds in $W^{1,p}\cap W^{2,\frac{p}{2}}$ for all $p<3$.  These are the first estimates which give $L_p$ bounds on the gradient with $p>2$, as well as first results providing control on second derivatives in general dimensions; compare \cite{SchoenUhlenbeck_RegularityHarmonic}, \cite{Simon_RegMin}.  In fact, the actual statement, which is much stronger, gives $L_p$ control for the {\it regularity scale} of a harmonic map; see Definition \ref{d:harm_reg_scale}.
In essence, we bound not just the volume of points in $M^n$ where $f$ does not have definite
derivative bounds, but also, the volume of points at which $f$ does not have such
derivative bounds on definite sized neighborhoods of these points. We will see that these estimates are  sharp.

In Corollary \ref{c:current_min_regularity} we give corresponding estimates for minimizing hypersurfaces.  Namely, we show that a minimizing hypersurface has second fundamental form $A$ lying in $L_p$ for all $p<7$.
This generalizes results of \cite{SchoenSimonYau_curvatureminimal} where such a bound was shown to hold for stationary minimal submanifolds with $p\leq 2+\epsilon(n)$.  Again, the actual statement of Corollary \ref{c:current_min_regularity} is much stronger, and gives $L_p$ control for the regularity scale of a current; see Definition \ref{d:current_reg_scale}.  Though we focus on currents here, the theorems are equally valid for varifolds.

\begin{remark}
\label{illrem}
 In the isolated singularity case, $n=8$, this result was independently proved by Tom Illamen.  In fact, in this case he proved the stronger statement that away from a {\it definite} finite number of points $\{p_\alpha\}$ there is the bound $|A|(x)\leq C \max |x-p_\alpha|^{-1}$.  Our ability to move from the isolated singularity case to the higher dimensional situation is based on Decomposition Lemma and the Cone-splitting Lemma; see Sections \ref{s:harm_Reg2Stratreduction}, \ref{s:current_Reg2Stratreduction} for more details.
\end{remark}

The proofs of the effective stratifications of Theorems \ref{t:harm_quant_strat} and \ref{t:current_quant_strat} are based on a quantitative version of blow up arguments (also referred to as "dimension reduction").
This  is a tool which uses the infinitesimal behavior of stationary maps and currents i.e. tangent maps and tangent cones, to obtain Hausdorff dimension estimates on singular sets $\cS$;
see \cite{Federer_GeomMeas}, \cite{SchoenUhlenbeck_RegularityHarmonic}.  Theorems \ref{t:harm_quant_strat} and \ref{t:current_quant_strat} exploit an additional principal,
"quantitative differentiation" (in the sense of \cite{CKN}, \cite{cheegernaber_quantstrateistein}, \cite{cheegerqdg}), in order to derive more effective Minkowski dimension
estimates.  These are estimates not just on  the singular sets themselves, but also on the volumes of tubes
$T_r(\cS)$ around the singular sets.  In addition,
what we call
the Decomposition Lemma and the Cone-splitting Lemma
are used to analyze the behavior of maps and currents at given fixed scales, rather than passing to a limit and studying tangential behavior.  This eventually yields the quantitative dimension reduction needed for Theorems \ref{t:harm_quant_strat} and \ref{t:current_quant_strat}.

The proofs of the regularity results of Theorems \ref{t:harm_min_regularity} and \ref{t:current_min_regularity}  require  new $\epsilon$-regularity theorems. These are given in Section \ref{s:harm_monotonicity}
(for harmonic maps) and Section \ref{s:current_monotonicity} (for minimizing currents).
The proofs are not difficult. Contradiction arguments are used
to reduce the statements to previously known results.
On the other hand, the theorems have a somewhat different character
from the $\epsilon$-regularity theorems of  \cite{Federer_GeomMeas}, \cite{SchoenUhlenbeck_RegularityHarmonic}.
Roughly speaking, these theorems assert that if a neighborhood of a point has small energy in the right sense, then the point is a smooth point.  By contrast, the $\epsilon$-regularity theorems of this paper state that if a neighborhood of a point has enough {\it approximate} degrees of symmetry, then the point is a smooth point.  Such $\epsilon$-regularity theorems are found  in riemannian geometry and particularly in the study of Einstein manifolds; see for instance \cite{MR1937830}.  The notion of {\it approximate} symmetry turns out to be exactly what can be controlled by the quantitative dimension reduction of Theorems \ref{t:harm_quant_strat} and \ref{t:current_quant_strat}.  Hence, when the $\epsilon$-regularity and quantitative stratification theorems are combined, we get the regularity results of Theorems \ref{t:harm_min_regularity}, \ref{t:current_min_regularity} and Corollaries \ref{c:harm_min_regularity}, \ref{c:current_min_regularity}.

 In  general outline, we will follow a scheme introduced in our paper,
\cite{cheegernaber_quantstrateistein}, where analogous estimates were obtained in the context of riemannian manifolds
with definite lower bounds on  Ricci curvature and on the collapsing; see Theorems 1.10, 1.17, 1.25.  The most refined results in that paper are for Einstein manifolds.  They give estimates on the ``curvature
radius'' off sets of small volume.  The "regularity scale" considered in the present paper is the analog of the curvature radius.  As in the present paper, a quantatitve differentiation theorem
in the sense of  \cite{CKN} (see also\cite{cheegerqdg})
plays a key role. Other features  introduced in \cite{cheegernaber_quantstrateistein}
which are also crucial here include   the Decomposition Lemma
 and the Cone-splitting Lemma; see Sections \ref{s:harm_Reg2Stratreduction} and \ref{s:current_Reg2Stratreduction}.

The paper is divided into two parts:  Part I concerns harmonic maps;
Part II concerns minimal currents.  Since there is a strong parallel
between the two cases,  in Part II, in so far as is possible, we will indicate where and how the harmonic map discussion can be modified to obtain the corresponding results on minimal currents.

\begin{remark}
\label{oe}

Although we concentrate on harmonic maps and minimal currents, the same techniques can be applied to similar nonlinear equations.  This will be discussed elsewhere.
The most straightforward applications would be to minimizers of other energy functionals; see for instance \cite{SchoenSimon_regularityminimal}.  Effective estimates on the nodal sets of harmonic maps should also be possible, see \cite{HanHardtLin_Nodal}.  Applications to nonlinear parabolic equations such as mean curvature flows and the Ricci flow also seem plausible but would require additional technical results.
\end{remark}

\part{Harmonic Maps}

\section{Main results on harmonic maps}
\label{smrh}
In this section we state our main quantitative results on harmonic
maps.  Specifically, we will be concerned with two classes of harmonic maps:  stationary
maps and minimizing maps.  For stationary harmonic maps, we will define a
 certain {\it quantitative} stratification of the singular set.  Our first main theorem, Theorem
\ref{t:harm_quant_strat}, is an estimate on  Minkowski content for the quantitative (or equivalently,
effective) strata.

Recall that the {\it Minkowski $r$-content} of a set $A$ is the number of closed metric
balls of radius $r$ in a minimal
covering of $A$.  In particular, if for all $r$, this number is bounded by $C_\eta\cdot r^{-(d+\eta)}$ then $A$ is said
to have  {\it Minkowski dimension $d$}.  Clearly, the Minkowski dimension is $\geq$ the Hausdorff dimension
since in the latter, coverings by balls of radius $r$ are replaced by (the larger class of) coverings by balls of radius $\leq r$. Throughout the paper, our notation convention is:
\begin{equation}
\begin{aligned}
\dim &={\rm Hausdorff\,\, dimension}\, ,\\
\dim_{{\rm Min}} &={\rm Minkowski\,\, dimension}\, .
\end{aligned}
\end{equation}

In view of the assumed sectional curvature bound, (\ref{con:sectional}),
in our situation (or more generally,
given a lower Ricci crvature bound) a
bound on the Minkowski \hbox{$r$-content} of a set yields a bound on the volume of its
$r$-tubular neighborhood.
Hence, depending on the
precise statement, an estimate on  Minkowski content provides can
either an effective version of  a Hausdorff
dimension estimate or a Hausdorff measure estimate. While both types of strengthened
estimates played a role in \cite{cheegernaber_quantstrateistein}, in the present paper only
former is relevant.

 Our principle application of this new quantitative stratification  to minimizing harmonic maps
is given in Theorem \ref{t:harm_min_regularity}.  There, when combined with appropriate $\epsilon$-regularity theorems, the quantitative stratification leads to bounds on the "regularity scale"; see Definition \ref{d:harm_reg_scale}.

\subsection{The standard stratification}
Prior to discussing the quantitative stratification of Theorem \ref{t:harm_quant_strat},
 we  will recall the standard stratification of the singular set for harmonic maps.  This is based on the notion of a
"$k$-homogeneous map".

\begin{definition}
Measurable map $h:\dR^n\rightarrow N^m$ is
{\it $k$-homogeneous} at $y\in\dR^n$ with respect to the
$k$-plane $V^k\subseteq\dR^n$ if:
\begin{enumerate}
 \item $h(y+z)=h(y+\lambda z)$ for every $\lambda>0$ and $z\in\dR^n$.
 \item $h(z)=h(z+v)$ for every $z\in\dR^n$ and $v\in V^k$.
\end{enumerate}
If $y= 0^n$ then we  say that $h$ is $k$-homogeneous.
\end{definition}

\noindent
Note that an $n$-homogeneous map is simply a constant map.

 For $y\in M$ and $0<r<\inj_{M^n}(y)$, define the map
$T_{y,r}f:B_{r^{-1}}(0^n)\subseteq \dR^n\rightarrow N^m$ by
$$
T_{y,r}f(z):= f\circ {\rm exp}_y(r z)\, .
$$
We call $T_yf:\dR^n\rightarrow N^m$ a {\it tangent map} of $f$ at $y$ if there
 exists $r_i\rightarrow 0$ such that
$$
\fint_{B_1(0^n)} \dist(T_yf(z),T_{y,r_i}f)^2\rightarrow 0\, .
$$
  Tangent maps at a point need not be unique. However, if $f$ is stationary, then
tangent maps are always $0$-homogeneous i.e. for every
$\lambda>0$,
 $$
T_yf(\lambda z)=T_yf(z)\, .
$$

 Now we define the natural stratification,
$$
\cS^0(f)\subseteq \cS^1(f)\subseteq\cdots\subseteq \cS^{n-1}(f)=\cS(f)\subseteq M^n\, .
$$
By definition, $y\in \cS^k(f)$ if and only if no tangent map $T_yf$ at $y$ is $(k+1)$-homogeneous.
By Schoen-Uhlenbeck, \cite{SchoenUhlenbeck_RegularityHarmonic} (see also \cite{white_stratification})
\begin{equation}
\label{de}
\dim \cS^k(f) \leq k\, .
\end{equation}
Moreover, by  \cite{SchoenUhlenbeck_RegularityHarmonic}, if $f$ is a minimizing harmonic map,
\begin{equation}
\label{cd3}
\begin{aligned}
\cS(f)&=\cS^{n-3}(f)\, ,\\
\dim\, \cS(f)&\leq n-3\, .
\end{aligned}
\end{equation}

\subsection{The quantitative stratification}
In order to define the quantitative version of this stratification, we first define the concept of
an ``almost $k$-homogeneous map''.

\begin{definition}
A measurable map $f:B_{2r}(x)\subseteq M^n\rightarrow N^m$ is {\it $(\epsilon,r,k)$-homogeneous} if there exists a
$k$-homogeneous map $h:\dR^n\rightarrow N^n$ such that
 $$
\fint_{B_1(0^n)}\dist(T_{x,r}f,h)^2<\epsilon\, .
$$
\end{definition}

 In the above  case, if $h$ is $k$-homogeneous with respect to the $k$-plane $V^k\subseteq\dR^n$,
 then we call a $V^k$ a {\it defining $k$-plane for $f$} and write
$$
V^k_{f,x}:= B_{r}(x)\cap \rexp_x(V^k)\, .
$$

Next we  introduce the  quantitative singular set.
\begin{definition}
 For each $\eta>0$ and $0<r<1$  the {\it $k^{th}$ effective
singular stratum $\cS^k_{\eta,r}(f)\subseteq M^n$} is the set
\begin{align}
 \cS^k_{\eta,r}(f):=\left\{y\in M^n: \fint_{B_1(0^n)}\dist(T_{y,s}f,h)^2> \eta\text{ for all }
r\leq s\leq 1\text{ and $(k+1)$-homogeneous maps }h\right\}.
\end{align}
\end{definition}

\noindent
Note that $y\in \cS^k_{\eta,r}(f)$ if and only if $f$ is not $(\eta,s,k+1)$-homogeneous for
every $r\leq s\leq 1$.  Moreover, it follows immediately from the definition that
\begin{equation}\label{containments}
\cS^{k}_{\eta,r}(f)\subset \cS^{k'}_{\eta',r'}(f)\qquad
({\rm if}\,\,k'\leq k,\,\,\eta'\leq\eta,\,\, r\leq r')\, ,
\end{equation}
\begin{equation}
\label{relation}
\cS^k(f)=\bigcup_\eta \,\,\bigcap_r \, \cS^k_{\eta,r}(f)\, .
\end{equation}

Our main theorem concerning the behavior of the effective singular
set is  Theorem \ref{t:harm_quant_strat} below.  It states that the known  Hausdorff dimension
estimates on the singular set can be strengthened to
estimates on the Minkowski
content of the quantitative stratification.  This is equivalent to
the (formally stronger) statement that the $r$-tubular neighborhoods
of the quantitative strata, $T_r(\cS^k_{\eta,r}(f))$, have volume which is
controlled by any power of the radius $r$ that is less than the Hausdorff codimension.
\begin{theorem}
\label{t:harm_quant_strat}
Let $f:B_2(x)\subseteq M^n\rightarrow N^m$ denote a stationary harmonic
map with bounded Dirichlet energy
\begin{equation}
\label{deb}
\int_{B_2(x)}|\nabla f|^2<\Lambda\, .
\end{equation}
 Then for all $\eta>0$ there exists $C=C(n,N^m,\Lambda,\eta)$,
such that  for any $0<r< 1$,
\begin{align}
\Vol(T_r(\cS^k_{\eta,r}(f))\cap B_1(x))\leq Cr^{n-k-\eta}.
\end{align}
\end{theorem}

\subsection{Quantitative estimates for the regularity scale}
In order to state the  consequences of Theorem \ref{t:harm_min_regularity}
 we first define the notion of the "regularity scale"
 of a function. This entails a refinement
of the notion of a pointwise $C^2$-bound.  A bound on the regularity
scale of $f$ at $x\in M^n$ controls the behavior of $f$ not just at $x$, but also
on a certain ball $B_{r_f(x)}(x)$.
Clearly, controlling $r_f(x)$ from below
is harder than controlling the pointwise $C^2$-norm
at $x$. Correspondingly, such control gives much stronger information.


Given any measurable map $f:M^n\rightarrow N^m$, put $r_{0,f}(x)=0$, if
$f$ is not $C^2$ in a neighborhood of $x$.  Otherwise define
$r_{0,f}(x)$ to be the maximum of $r>0$ such that $f$ is $C^2$ on $B_r(x)$.

\begin{definition}
\label{d:harm_reg_scale}
Define the regularity scale $r_f(x)$ by
\begin{align}
 r_f(x):=\max\left\{0\leq r\leq r_{0,f}(x): \sup_{B_r(x)}r|\nabla f|+r^2|\nabla^2 f|\leq 1\right\}.
\end{align}
\end{definition}

Note that the quantity whose supremum is
being taken  is a scale invariant quantity. Therefore, if $r_f(x)=r$ and
we rescale $B_r(x)$ to a ball of unit size and view $f:B_1(x)\rightarrow N^m$,
then $|\nabla f|+|\nabla^2 f|\leq 1$ on $B_1(x)$.
Also observe that if $f$ is a weakly harmonic map then, by standard elliptic regularity, if $r_f(x)\geq r$ then for all $k\in \dZ+$,
\begin{equation}
\label{ee}
\sup_{B_{\frac{r}{2}}(x)}r^k|\nabla^{k} f|\leq C_k\, ,
\end{equation}
where the constant $C_k$ depends possibly on the curvature and derivatives of the curvature on both $M^m$ and $N^m$.  In particular, a lower bound on the regularity scale at a point gives bounds for all derivatives of a weakly harmonic map in a definite sized neighborhood of that point.

Next we partition $M^n$ into good and bad sets based on the behavior of $f$.
\begin{definition}
Given any measurable map $f:M^n\rightarrow N^m$ and any $r>0$ we define
\begin{align}
 \cB_{r}(f):= \{x\in M: r_{f}(x)\leq r\}.
\end{align}
\end{definition}

The following is the principle application of our main theorem. It strengthens the
 Hausdorff dimension estimates on the singular set of a minimizing harmonic map
which were given in \cite{SchoenUhlenbeck_RegularityHarmonic}, to
corresponding lower bounds on the regularity scale off sets of appropriately small
volume.
\begin{theorem}
\label{t:harm_min_regularity}
Let $f:B_2(x)\subseteq M^n\rightarrow N^m$ denote a minimizing harmonic map with
bounded Dirichlet energy as in (\ref{deb}).
  Then for all $\eta>0$, there exists $C=C(n,N^m,\Lambda,\eta)$,
such that  for any $0<r< 1$,
\begin{enumerate}
 \item $\Vol(T_r(\cB_{r}(f))\cap B_{1}(x))\leq Cr^{3-\eta}$.

In particular, for minimizing harmonic maps, we get the Minkowski dimension
bound
\begin{equation}
\label{md}
\dim_{{\rm Min}}\cS(f)\leq n-3\, .
\end{equation}
 \item Moreover, if $N^m$ is such that for all $1\leq \ell\leq k$ there
 exists no smooth minimizing harmonic map $s:S^\ell\rightarrow N^m$,
then $\Vol(\cB_{r}(f)\cap B_{1}(x))\leq Cr^{3+k-\eta}$.
\end{enumerate}
\end{theorem}

The following stronger consequence follows directly from Theorem \ref{t:harm_min_regularity} and elliptic regularity
theory.
\begin{corollary}\label{c:harm_min_regularity}
 Let $f:B_2(x)\subseteq M^n\rightarrow N^m$ denote a minimizing harmonic map with Dirichlet
 energy bounded as in (\ref{deb}).  Then:
\begin{enumerate}
 \item For every $0<p<3$ there exists $C=C(n,N^m,\Lambda,p)$ such that
$$
\int_{B_1(x)}|\nabla f|^p\leq \int_{B_1(x)}r_{f}^{-p}<C\, .
$$
 \item More generally, there exists $C=C(n,N^m,\Lambda,p)$
such that if $N^m$ is such that for all $2\leq \ell\leq k$ there exists
 no smooth minimizing harmonic map, $s:S^\ell\rightarrow N^m$, then for every
$0<p<2+k$,
$$
\int_{B_1(x)}|\nabla f|^p\leq \int_{B_1(x)}r_{f}^{-p}<C\, .
$$
\end{enumerate}
Moreover (by elliptic regularity) there exists $C=C(n,N^m,\Lambda,p,k)$ such that for $2<p$ as above,
\begin{equation}
\label{er}
\int_{B_{1}(x)}|\nabla^2 f|^{\frac{p}{2}}<C\, .
\end{equation}
\end{corollary}

\begin{remark}
\label{rsharp}
 The above $L^p$ estimates are sharp.  For instance, for each $k\geq 1$, consider the map
$$
f:B_1(0^{k+2})\rightarrow S^{k+1}\, ,
$$
where  $S^{k+1}$ is the unit $(k+1)$-sphere, such that $f$ is an isometry on the boundary and
constant in the radial direction.  This map is a minimizing harmonic map with $N^m=S^{k+1}$ satisfying the condition that there are no smooth minimizing harmonic maps from $S^\ell$ into $N^m$ for all $2\leq \ell\leq k$.  The gradient $|\nabla f|$ can easily be checked to lie in $L^p$ for all $p<2+k$. However, it fails to lie in $L^{2+k}$.
\end{remark}

We will focus primarily on giving complete details of the proofs in the case $M^n=\dR^n$.
 Since the proof in the
general case is essentially the same (up to the appearance of some additional insignificant constants)
 for the general case, we will just give additional comments as needed.

\section{Preliminaries and reduction of main results to  Theorem \ref{t:harm_quant_strat}}
\label{s:harm_prelim}
In this section, we establish some preliminary results
that are required for the proofs of
our main results. These were stated in Section \ref{smrh}. The
preliminaries are  counterparts of  results which played an analogous role in
\cite{cheegernaber_quantstrateistein}.  In that case, the preliminary results,
though less routine, were
already known.

The {\it quantitative rigidity theorem}, Theorem \ref{t:harm_monotone_rigidity},  corresponds to the ``almost volume cone implies almost metric cone" theorem of
\cite{cheegercoldingalmostrigidity}, for manifolds with lower Ricci curvature bounds.  This was used in the proof of the quantitative differentiation theorem of \cite{cheegernaber_quantstrateistein}. A
similar role is played here, by Theorem \ref{t:harm_monotone_rigidity}. The
\hbox{$\epsilon$-regularity} theorem,
 Theorem \ref{t:harm_eps_regularity}, is used in combination with the quantitative stratification bound
given in Theorem \ref{t:harm_quant_strat}, to obtain our main result for the regularity scale for harmonic maps;
see Theorem \ref{t:harm_min_regularity}.  This also parallels the discussion of \cite{cheegernaber_quantstrateistein}.

\subsection{Monotonicity, quantitative rigidity and $\epsilon$-regularity}\label{s:harm_monotonicity}
In this subsection we consider maps $f:B_2(0^n)\subseteq\dR^n\rightarrow N^m$.
Given such an $f$, we  define for $x\in B_1(0^n)$ and $0<r<1$, the normalized Dirichlet functional
\begin{align}
\label{thetadef}
\theta_r(x):= r^{2-n}\int_{B_r(x)}\left |\nabla f\right |^2.
\end{align}
It is well known that if $f$ is a stationary harmonic map, then for each $x$ the function
$\theta_r(x)$ is monotone nondecreasing in $r$ and satisfies the monotonicity formula:
\begin{equation}
\label{thetamonotone}
\theta_r(x)-\theta_s(x)\, = \int_s^r\int_{\partial B_t(x)}t^{2-n}\left|\frac{\partial f}{\partial t}\right |^2\, .
\end{equation}
In particular,
$$
\theta_r(x)\uparrow\, ,
$$
and
$\theta_s(x)=\theta_r(x)$ if and only if $f$ is radially constant on the annulus $A_{s,r}(x)$; see  \cite{HardtWolf_Nonlinear},\cite{SchoenUhlenbeck_RegularityHarmonic} .

The following quantitative rigidity theorem is a direct consequence of the above,
together with a contradiction argument.

\begin{theorem}[Quantitative rigidity]
\label{t:harm_monotone_rigidity}
Let $f:B_2(0^n)\rightarrow N^m$ denote a stationary harmonic map with Dirichlet energy bounded
as in (\ref{deb}). Then
for every $\eta>0$, there exists $\epsilon=\epsilon(n,N^m,\Lambda,\eta)>0$,
$r=r(n,N^m,\Lambda,\eta)>0$,
such that if
\begin{align}
\theta_1(0^n)-\theta_r(0^n)\leq \epsilon,
\end{align}
then $f$ is $(\eta,1,0)$-homogeneous at $0^n$.
\end{theorem}
\begin{proof}
Assume the statement is false for some $\eta>0$.  Then there exists a sequence of stationary harmonic maps $f_i:B_2(0^n)\rightarrow N^m$ with energy bounded as in (\ref{deb}) with
$\theta_1(0^n)-\theta_{i^{-1}}(0^n)\leq i^{-1}$,
 but such that the $f_i$ are not $(\eta,1,0)$-homogeneous.  After passing to a subsequence we can take
$f_i\rightarrow f$,
where $f:B_2(0^n)\to N^m$ and the convergence is weak in $H^1$ and {\it strong} in $L^2$.  It follows from
 the weak convergence and the monotonicity formula (\ref{thetamonotone}), that $f$ is itself constant in the radial direction.
  In particular, $f$ is $0$-homogeneous, and by the {\it strong} convergence in $L^2$ it
follows that for $i$ sufficiently large the $f_i$ are $(\eta,1,0)$-homogeneous, a contradiction.
\end{proof}

\begin{remark}
\label{ca}
Contradiction arguments of the above type are applicable in many cases in which
one wishes to promote a rigidity theorem to a quantitative rigidity theorem
(as is required when proving a quantitative differentiation theorem).
  However, it is not known whether this sort of argument is applicable in the context of
the "almost volume cone implies almost metric cone" theorem of \cite{cheegercoldingalmostrigidity}.
\end{remark}

\begin{remark}
We also note that by their nature, contradiction arguments do not give explicit dependence of the relevant constants (such as $\epsilon,\,r$ above) on the small parameter ($\eta$ above). Even in cases in which a direct argument giving such dependence is possible, it may be extremely tedious and may involve additional technical difficulties.    While for geometric analytic applications like those discussed here, such dependence is not required, for the application to the "sparsest cut problem" given in \cite{CKN}, it is crucial. There, obtaining the desired estimate for this dependence is by far the most technically difficult part of the argument.
\end{remark}


The following $\epsilon$-regularity theorem is not based on the usual
small energy assumption, but rather on the almost homogeniety of the
minimizing harmonic map.  Roughly speaking, it states that a minimizing
harmonic map with sufficient approximate symmetry must be smooth.

\begin{theorem}[$\epsilon$-regularity]
\label{t:harm_eps_regularity}
Let $f:B_{2}(0^n)\rightarrow N^m$ denote a minimizing harmonic map with bounded
Dirichlet energy as in (\ref{deb}).
Then there exists
 $\epsilon=\epsilon(n,N^m,\Lambda)>0$ such that
\begin{enumerate}
\item If $f$ is $(\epsilon,2,n-2)$-homogeneous then
$$
r_f(0)\geq 1\, .
$$
\item Moreover, if $N^m$ is such that for all $2\leq \ell\leq k$ there exists no
smooth minimizing harmonic map $s:S^\ell\rightarrow N^m$ and $f$
is $(\epsilon,2,n-k-1)$-homogeneous then
$$
r_f(0)\geq 1\, .
$$
\end{enumerate}
\end{theorem}
\begin{proof}
Since the proofs of the two statements are identical,  we focus only on the first statement.

First, note that it is a consequence of \cite{HardtWolf_Nonlinear},
\cite{SchoenUhlenbeck_RegularityHarmonic}, that the result holds for $n$-homogeneous maps.  Namely, there exists $\epsilon=
\epsilon(n,N^m,\Lambda)$ such that if for some $w\in N^m$,
$$
\fint_{B_2(0^n)}{\rm dist}(f,w)^2<\epsilon\, ,
$$
then
$$
\sup_{B_1(0^n)}\left(|\nabla f|+|\nabla^2 f|\right)\leq 1\, .
$$

We now show that for appropriate $\epsilon$,
the statement holds for $(\epsilon,2,n-2)$-homogeneous maps.  So assume
otherwise.  Then there exists a sequence $\epsilon_i\leq\frac{1}{i}$ and minimizing harmonic maps
$f_i:B_{2}(0^n)\rightarrow N$ with
$$
\int_{B_{2}(0)}|\nabla f_i|^2<\Lambda\,
$$
such that the $f_i$ are $(\epsilon_i,2,n-2)$-homogeneous, but $r_{f_i}(0)<1$.  After passing to a
subsequence we get
$
f_i\xrightarrow{L^2}f\, ,
$
where $f:B_2(0^n)\rightarrow N^m$ is now a minimizing harmonic map which is $(n-2)$-homogeneous on
$B_2(0)$ and for which the Dirichlet energy bound
(\ref{deb})
holds.  It is shown in \cite{SchoenUhlenbeck_RegularityHarmonic} that such a map is necessarily a constant: $f\equiv z$.  In particular, since the $f_i$ are converging strongly in $L^2$, for $i$ large we conclude that
$$
\fint_{B_2(0)}{\rm dist}(f_i,w)^2<\tilde \epsilon_i\, ,
$$
where $\tilde \epsilon_i\rightarrow 0$.  For $\tilde \epsilon_i<\epsilon(n,N^m,\Lambda)$ as in the beginning of the proof this yields $r_{f_i}(x)\geq 1$, a contradiction.
\end{proof}

\begin{remark}
  In the general case, $f:M^n\rightarrow N^m$,
the assumption that $f$ is $(\epsilon,2,n-2)$-homogeneous should
be replaced by the
assumption that $f$ is $(\epsilon,2r,n-2)$-homogeneous, where
$r\leq r(n,N^m,\Lambda)$.  Then,  the above
contradiction argument with $\epsilon_i,r_i\rightarrow 0$ can be repeated.
After blow up, we are reduced to the case $M^n=\dR^n$.
\end{remark}

\subsection{Reduction of Theorem \ref{t:harm_min_regularity} to Theorem \ref{t:harm_quant_strat}}
\label{s:harm_Reg2Stratreduction}

In  this subsection, we show that Theorem \ref{t:harm_min_regularity} follows from the rigidity and $\epsilon$-regularity theorems,
together with Theorem \ref{t:harm_quant_strat},
the quantitative stratification theorem.  The remaining two sections of Part I  will
 be devoted to proving Theorem \ref{t:harm_quant_strat}.

\begin{proof}[Proof of Theorem \ref{t:harm_min_regularity}]
Since the proofs of the first and second statements of the theorem are essentially identical, we focus on the first statement.
 As with the other main results of this paper, we will restrict attention
to the case $M^n= \dR^n$; the general case is the same up to some additional lower order errors.

It follows immediately from  Theorem \ref{t:harm_eps_regularity}  that if $\eta\leq \eta(n,N^m,\Lambda)$ with $x\in \cS^{n-3}_{\eta,r}(f)$, then $x\in \cB_{r}(f)$. In particular, for all $\eta$ sufficiently small, we have
$$
T_r(\cB_r(f))\subseteq T_r(\cS^{n-3}_{\eta,r})\, .
$$
Hence,
\begin{align}
\Vol(T_r(\cB_r(f))\cap B_1(x))\leq \Vol(T_r(\cS^{n-3}_{\eta,r}))\leq C(n,N^m,\Lambda,\eta)r^{3-\eta}\, ,
\end{align}
as claimed.
\end{proof}

\section{Reduction of Theorem \ref{t:harm_quant_strat} to the Covering Lemma}
\label{s:harm_decomposition}

Although the title of this section refers to the Covering Lemma, an equally important
role is played by the Decomposition Lemma, Lemma \ref{l:harm_decomposition}.

In outline we proceed as follows.  We begin by stating the Decomposition Lemma. This lemma
has two items in its statement.  Using this Lemma we observe that Theorem \ref{t:harm_quant_strat} is virtually an immediate consequence.  Next, we prove  item 1. of the Decomposition Lemma.  In particular,  this involves a
quantitative differentiation argument in the sense of \cite{CKN}, \cite{cheegernaber_quantstrateistein}, \cite{cheegerqdg}.  Finally, we state the Covering Lemma, Lemma \ref{l:harm_splitting}, and observe that item 2. of the Decomposition Lemma is a simple consequence of the Covering Lemma.
(For further explanation of the relationship between the Covering Lemma
and the Decomposition Lemma, see Remark \ref{significance}.)
 The proof of the Covering Lemma
is given in Section \ref{s:harm_cone_splitting}.

\subsection{The Decomposition Lemma and the Proof of Theorem \ref{t:harm_quant_strat}}
Roughly speaking, the Decomposition Lemma states that
set $\cS^k_{\eta,\gamma^j}(f)$ can be covered by a collection of {\it nonempty} sets,
$\{\cC^k_{\eta,\gamma^j}\}$, each of which consists of a not too large
collection of balls of radius $\gamma^j$.  The sets $\{\cC^k_{\eta,\gamma^j}\}$ themselves are formed by decomposing $\cS^k_{\eta,\gamma^j}(f)$ based on the behavior of points at various scales, see (\ref{e:harm_HL}) and (\ref{Edef}).  The cardinality of the collection
$\{\cC^k_{\eta,\gamma^j}\}$  goes to infinity as $j\to \infty$. However,
 according to Lemma \ref{l:harm_decomposition}, the growth rate of the number of
 sets in $\{\cC^k_{\eta,\gamma^j}\}$ is bounded by $\leq j^{K(n,\eta,\Lambda)}$.
This turns out to be  slow enough to be negligible our purposes.

\begin{lemma}[Decomposition Lemma]
\label{l:harm_decomposition}
There exists $c_1(n), c_0(n),\,K(n,\Lambda,\eta,\gamma,N),\, Q(n,\Lambda,\eta,\gamma,N)>0$ such that
for each $j\in\dZ_+$:

\begin{enumerate}
\item The set $\cS^k_{\eta,\gamma^j}(f)\cap B_1(0)$ is contained in the union of at most
$j^K$ nonempty sets $\cC^k_{\eta,\gamma^j}$.
\item Each set $\cC^k_{\eta,\gamma^j}$  is the union of at most
$(c_1\gamma^{-n})^{Q}\cdot (c_0\gamma^{-k})^{j-Q}$ balls of radius $\gamma^j$.
\end{enumerate}
\end{lemma}

Next we show that   Lemma \ref{l:harm_decomposition}  implies Theorem \ref{t:harm_quant_strat}.
\begin{proof}[Proof of Theorem \ref{t:harm_quant_strat}]
Clearly, it suffices to verify Theorem \ref{t:harm_quant_strat} for $r$ of the form
$\gamma^j$ for some convenient choice $\gamma=\gamma(n,\eta)<1$.
Given Lemma \ref{l:harm_decomposition},
an appropriate choice is
$$
\gamma=c_0^{-\frac{2}{\eta}}\, .
$$

The volume of a ball satisfies
\begin{equation}
\label{e:harm_volume}
\Vol(B_{\gamma^j}(x))= w_n\gamma^{jn}\, ,
\end{equation}
which together with
$$
c_0^j\leq (\gamma^j)^{-\frac{\eta}{2}}\, ,
$$
$$
j^K\leq c(n,\Lambda)(\gamma^j)^{-\frac{\eta}{2}}\, ,
$$
 gives
\begin{equation}
\label{e:main}
\begin{aligned}
\Vol(\cS^k_{\eta,\gamma^j}\cap B_1(0))
&\leq j^K\cdot \left[  (c_1\gamma^{-n})^{Q} \cdot  (c_0\gamma^{-k})^{j-Q} \right]
 \cdot w_n\cdot(\gamma^j)^n\\
&\leq c(n,K,Q) \cdot j^K\cdot c_0^{j}\cdot(\gamma^j)^{n-k}\\
&\leq c(n,K,Q)\cdot(\gamma^j)^{n-k-\eta}\, .
\end{aligned}
\end{equation}

 From the above, for all $r\leq1$, we get
$$
 \begin{aligned}
 \Vol(\cS^k_{\eta,r}\cap B_1(\underline{x}))
&\leq \gamma\cdot c(n,K,Q)\cdot r^{n-k-\eta}\\
& \leq c(n,\eta,\Lambda,N) r^{n-k-\eta}\, .
\end{aligned}
$$
Therefore, modulo the proof of  Lemma \ref{l:harm_decomposition},
the proof Theorem \ref{t:harm_quant_strat} is complete.
\end{proof}
\subsection{Construction of the decomposition}
We begin with the definition of the sets in the collection $\{\cC^k_{\eta,\gamma^j}\}$.  To this end, we introduce
a quantity $\cN_t(f,B_r(x))\geq 0$, the
"$t$-nonhomogeneity" of a ball $B_r(x)$. This quantity measures how far $f$ is
from being \hbox{$0$-homogeneous} on $B_r(x)$.

\begin{definition}
Let $x\in B_1(0)$ with $0<r<1$ and $t\geq 1$.  Then we define {\it $t$-nonhomogeniety}
$\cN_t(f,B_r(x))$
 as the infimum of $\psi>0$ such that $f$ is $(\psi,tr,0)$-homogeneous.
\end{definition}

Given $\epsilon>0$,  we can break up $B_1(0)$ into the following subsets.
\begin{equation}
\label{e:harm_HL}
\begin{aligned}
H_{t,r,\epsilon}(f) & =\{x\in B_1(0)\, |\,   \cN_t(f,B_r(x))\geq \epsilon\}\, , \\
L_{t,r,\epsilon}(f) &=\{x\in B_1(0)\, |\,   \cN_t(f,B_r(x))< \epsilon\}\,  .
\end{aligned}
\end{equation}

The construction which follows makes sense for arbitrary $\epsilon>0$. Note however,
that the statement of the Decomposition Lemma does not involve
a choice of $\epsilon>0$. Therefore,  we now fix
\begin{equation}
\label{efix}
\epsilon=\epsilon(n,\eta,\gamma)\, ,
\end{equation}
where $\epsilon(n,\eta,\gamma)$ is as in the Covering Lemma,
Lemma \ref{l:harm_splitting} .

To each point $x$, we associate a $j$-tuple $T^j(x)$ as follows. By definition, for all $i\leq j$,
the $i$-th entry of $T^j(x)$ is $1$ if $x\in H_{\gamma^{-n},\gamma^i,\epsilon}$
 and $0$ if $x\in L_{\gamma^{-n},\gamma^i,\epsilon}$.  Then, for each $j$-tuple $T^j$, we  put
\begin{equation}
\label{Edef}
E_{T^j}=\{x\in B_1(0^n)\, |\, T^j(x)=T^j \}\, .
\end{equation}

 Let $\cC^k_{\eta,\gamma^0}(T^0)\equiv B_1(0)$ and by definition let $T^{j-1}$ be the $j-1$-tuple
obtained from $T^j$ by dropping the last entry.  Assume that the
nonempty subset $\cC^k_{\eta,\gamma^{j-1}}(T^{j-1})$ has been
defined and satisfies item 2. of  Lemma \ref{l:harm_decomposition}.
Assume in addition, that
$\cC^k_{\eta,\gamma^{j-1}}(T^{j-1})\supset \cS^k_{\eta,\gamma^j}\cap  E_{T^j}$.
\vskip2mm

\noindent
{\bf Induction step.}
 For each ball $B_{\gamma^{j-1}}(x)$ of $\cC^k_{\eta,\gamma^{j-1}}(T^{j-1})$,
take a minimal covering of $B_{\gamma^{j-1}}(x)\cap \cS^k_{\eta,\gamma^j}\cap  E_{T^j}$
by balls of radius $\gamma^j$ with centers in
$B_{\gamma^{j-1}}(x)\cap \cS^k_{\eta,\gamma^j}\cap  E_{T^j}$.
Define the union of all balls so obtained to be
$\cC^k_{\eta,\gamma^j}(T^{j})$, {\it provided it is nonempty}.
\vskip2mm

\subsection{Proof of item 1. of the Decomposition Lemma}
\label{hdl}

A priori, because the sets $\cC^k_{\eta,\gamma^j}(T^{j})$ are indexed by $j$-tuples with values $0$ or $1$, there are $2^j$ nonempty such subsets.  However, as we will show below, there exists $K=K(n,\Lambda,\eta,\gamma,N)$ such that
\begin{align}
\label{e:harm_badscales}
E_{T^j}\neq \emptyset\,\, {\rm implies}\,\, |T^j|\leq K(n,\Lambda,\eta,\gamma,N)\, .
\end{align}
  Since the number of $T^j$ with $|T^j|< K$
is at most
\begin{equation}
\label{e:harm_possibilities}
 {j\choose K} \leq j^{K}\,  ,
\end{equation}
it will follow that the cardinality of $\{\cC^k_{\eta,\gamma^j}(T^j)\}$ is at most $ j^K$.
 Thus the  item 1. of the Decomposition Lemma follows from (\ref{e:harm_badscales}).

Next, we verify (\ref{e:harm_badscales}).
Let the notation be as in Section \ref{s:harm_monotonicity}. For $r>0$,  we consider the
 normalized Dirichlet energy $\theta_r(x)$, defined in (\ref{thetadef}).  Recall that
by (\ref{thetamonotone}), $\theta_r(x)$
is a nonincreasing function of $r$.  Moreover,  $\theta_s(x)=\theta_r(x)$ if and only if
 $f$ is radially constant on the annulus $A_{s,r}(x)$.

For $s<t$, define the  {\it $(s,t)$-Dirichlet energy} $\cW_{s,t}(x)$ by
$$
\cW_{s,t}(x):= \theta_t(x)-\theta_s(x)\geq 0\, .
$$
Note that if $t_1\leq s_2$ then
\begin{equation}
\label{e:additivity}
\cW_{s_1,t_2}(x)\leq \cW_{s_1,t_1}(x)+\cW_{s_2,t_2}(x)\, ,
\end{equation}
with equality if $t_1=s_2$.

Let $(s_i,t_i)$ denote a possibly infinite sequence of intervals with $t_{i+1}\leq s_i$ and $t_1=1$.  Using the assumed $\Lambda$-bound on the full Dirichlet energy, we can write
\begin{equation}
\label{e:eb}
\cW_{t_1,s_1} + \cW_{t_2,s_2}+\cdots \leq \Lambda \, .
\end{equation}
where the terms on the left-hand side are all nonnegative.

Now fix $\delta>0$ and let $N$ denote the number of $i$ such that
$$
\cW_{\gamma^i,\gamma^{i-n}}>\delta\, .
$$
Then
\begin{equation}
\label{e:Nb}
N\leq (n+1)\cdot\delta^{-1}\cdot \Lambda\, .
\end{equation}
Otherwise, there would be at least $\delta^{-1}\cdot \Lambda$ {\it disjoint} closed intervals of the form $[\gamma^i,\gamma^{i-n}]$ with $\cW_{\gamma^{i}, \gamma^{i-n}}>\delta$, contradicting (\ref{e:eb}).

Lemma \ref{t:harm_monotone_rigidity} implies the existence of $\delta=\delta (n,\Lambda,\epsilon)$ such that if $\cW_{\gamma^{i}, \gamma^{i-n}}\leq \delta$ then $\cN_{\gamma^{-n}}(B_{\gamma^i}(x))\leq \epsilon\gamma^i$, i.e. $x\in L_{\gamma^{-n},\gamma^i,\epsilon}$.  Since $\epsilon=\epsilon(n,\eta,\gamma)$ has
been fixed as in (\ref{efix}), this gives (\ref{e:harm_possibilities}), i.e.
$|T^j|<N\leq K(n,\eta,\Lambda,\gamma,N)$ if $E_{T^j}\neq \emptyset$.

As as noted just after  (\ref{e:harm_possibilities}), this implies item 1. of the Decomposition Lemma \ref{l:harm_decomposition}.

\begin{remark}
Clearly, (\ref{e:harm_badscales}) is the quantitative version of the fact that tangent maps are $0$-homogeneous.
The argument we have given is an instance of {\it quantitative differentiation} in the sense of Section 14 of \cite{CKN}.
\end{remark}

\subsection{Reduction of  item 2.  of the Decomposition Lemma to the Covering Lemma}

Recall that the sets $\cC^k_{\eta,\gamma^{j-1}}\cap E_{T^j}$ are constructed
inductively, using minimal coverings of balls of radius
$\gamma^{j-1}$ by balls of radius $\gamma^j$; see the {\bf Induction step} just prior to (\ref{e:harm_badscales}).
Also note that in view of the doubling condition on the riemannian measure,
 any ball $B_{\gamma^{j-1}}(x)$ can be covered by
 at most
$c_1(n)\gamma^{-n}$ balls of radius $\gamma^j$.  However, when $j>n$ and the $j$-th entry of $T^j$
is $0$, we use instead the following lemma, whose proof will be given
in Section \ref{s:harm_cone_splitting}.

\begin{lemma}[Covering lemma]
\label{l:harm_splitting}
There exists $\epsilon=\epsilon(n,\eta,\gamma)$, such that if
$\cN_{\gamma^{-n}}(B_{\gamma^{j-1}}(x))\leq \epsilon$ and
$B_{\gamma^{j-1}}(x)$ is a ball of $\cC^k_{\eta,\gamma^{j-1}}(T^{j-1})$,
 then the number of balls in the minimal covering of
$B_{\gamma^{j-1}}(x) \cap \cS^k_{\eta,\gamma^j}\cap L_{\gamma^{-n},\gamma^j,\epsilon}$
 is $\leq c_0(n)\gamma^{-k}$.
\end{lemma}

Assuming Lemma \ref{l:harm_splitting}, an obvious induction argument yields
the bound on the number of balls of $\cC^k_{\eta,\gamma^j}$ appearing in item 2.
of Lemma \ref{l:harm_decomposition}.

\begin{remark}
\label{significance}
Note that the above induction argument relies crucially on the fact that we can
deal separately with each  individual set, $\cC^k_{\eta,\gamma^j}$,
all of whose points have the sequence of good and bad scales,
and then sum over the collection of all such subsets.
Absent the decomposition into such subsets, it is not at all clear
how such an induction argument could be carried out.
\end{remark}

\begin{remark}
Lemma \ref{l:harm_splitting} can be viewed as the quantitative analog of the density
argument in the proof that $\dim \cS^k(f)\leq k$; see (\ref{de}).
\end{remark}

The proof of Lemma \ref{l:harm_splitting}
will be a direct consequence of Corollary \ref{c:harm_inductive} of the Cone-splitting
Lemma, Lemma \ref{l:harm_cone_splitting}.

\section{Proof of the Covering Lemma via the Cone-splitting Lemma}
\label{s:harm_cone_splitting}

A cone-splitting principle for tangent cones at points of limit spaces with a  definite lower Ricci
curvature bound,
together with its quantitative
refinement, figured prominently  in \cite{cheegernaber_quantstrateistein}. An analogous cone-splitting principle for harmonic
maps  plays a key role here.
\vskip2mm

\noindent
{\bf Cone-splitting principle of} \cite{cheegernaber_quantstrateistein}{\bf .}
Let $C(X), \, C(Y)$ denote {\it metric cones} and assume that the diameter of $X$ is at most $\pi$.
Let $x^*,\, y^*$ denote vertices of $C(X), C(Y)$.
If there exists an isometry $h:C(Y)\to\dR^k\times C(X)$   for which
 $h(y^*)\not\in \dR^k\times C(X)$,
then $\dR^k\times C(X)$ is isometric to some cone $\dR^{k+1}\times C(Z)$.
\vskip1mm

\noindent
{\bf Cone-splitting principle for harmonic maps.}
 If $h:\dR^n\rightarrow N^m$ is $k$-homogeneous at $y$ with respect to the $k$-plane $V^k$,
$h$ is $0$-homogeneous at $z\in\dR^n$, and $z\not\in y+V$, then
$h$ is $(k+1)$-homogeneous at $y$ with respect to the $(k+1)$-plane ${\rm span}\{z-y,V^k\}$.
\vskip2mm

 From the above and a contradiction argument as in Theorem \ref{t:harm_monotone_rigidity}, we immediately obtain the following quantitative refinement.

\begin{lemma}[Cone-Splitting Lemma]
\label{l:harm_cone_splitting}
There exists $\delta=\delta(n,N^m,\Lambda,\epsilon,\tau,\gamma)>0$ with the following property.
Let $\epsilon,\tau,\gamma>0$ and let
$f:B_{\gamma^{-1}}(0^n)\rightarrow N$ denote a stationary harmonic map with bounded Dirichlet energy,
$$
\int_{B_{\gamma^{-1}}(0^n)}|\nabla f|^2<\Lambda\, .
$$
Assume in addition:
\begin{enumerate}
 \item $f$ is $(\delta,\gamma^{-1},k)$-homogeneous at $0^n$.
 \item There exists $y\in B_1(0^n)\setminus T_{\tau}(V^k_{f,0^n})$ such that $f$ is $(\delta,2,0)$-homogeneous at $y$.
\end{enumerate}
Then $f$ is $(\epsilon,1,k+1)$-homogeneous at $0^n$.
\end{lemma}

The import of Lemma \ref{l:harm_cone_splitting} can be paraphrased as stating:
When points are well behaved in the sense of (\ref{e:harm_HL}) and lie near one another, then they interact and force a surrounding neighborhood to inherit even better properties.

The notation of the next corollary is as in Section \ref{s:harm_decomposition}.

\begin{corollary}
\label{c:harm_inductive}
For all $\epsilon,\tau,\gamma >0$ there exists $\delta(n,N^m,\Lambda,\epsilon,\tau,\gamma)>0$ and $\theta(n,N^m,\Lambda,\epsilon,\tau,\gamma)>0$ such that the following holds.
Let $f:B_{2}(0^n)\rightarrow N$ denote a stationary harmonic map with bounded Dirichlet energy
as (\ref{deb}).
 Let $r\leq \theta$ and $x\in L_{\gamma^{-n},r,\delta}(f)$. Then there exists $0\leq \ell\leq n$ such that
\begin{enumerate}
\item $f$ is $(\epsilon,r,\ell)$-homogeneous,
\item $L_{\gamma^{-n},\delta,r}\cap B_r(x)\subseteq T_{\tau r}(V^\ell_{f,x})$ .
\end{enumerate}
\end{corollary}

\begin{proof}
 For $\delta_0(n,N^m,\Lambda,\epsilon,\tau,\gamma^n)$ as in Lemma \ref{l:harm_cone_splitting}, inductively define
$\delta^{[n-i]}=\delta_0\circ\delta_0\circ\cdots\circ\delta_0$
\hbox{($i$ factors in the} composition). Then $\delta^{[0]}<\delta^{[1]}<\cdots<\delta_0$ and let us put $\delta=\delta^{[0]}$.  Since by assumption, $x\in L_{\gamma^{-n},r,\delta}$,  there exists a largest $\ell\geq 0$ such that $f$ is $(\epsilon,r,\ell)$-homogeneous.  To see that the corollary holds for this value of $\ell$, apply
Lemma \ref{l:harm_cone_splitting} to the rescaled ball $B_{\gamma^{-(n-\ell-1)}r}(x)$, with $\delta$ as above.
\end{proof}

Now we can finish the proof of the covering lemma.

\begin{proof}[Proof of Lemma \ref{l:harm_splitting}]
Let $B_{\gamma^{j-1}}(x)$ be as in Lemma \ref{l:harm_splitting}.  Since by assumption, $x\in \cS^k_{\eta,\gamma^i}\cap L_{\gamma^{-n},\gamma^j,\epsilon}$ we have that $f$ is not $(\eta,\gamma^{-n}\gamma^{j-1},k+1)$-homogeneous at $x$.  In particular by applying Corollary \ref{c:harm_inductive} with $\epsilon=\frac{1}{10}\gamma$ it follows that for some $\ell\leq k$ that $L_{\gamma^{-n},\gamma^j,\epsilon}\cap \cS^k_{\eta,\gamma^j}\cap B_{\gamma^{j-1}}(x)\subseteq T_{\frac{1}{10}\gamma^j}(V^\ell_{f,x})\cap B_{\gamma^{j-1}}(x)$, from which the result follows.
\end{proof}

\part{Minimal Currents}

\section{Main results on minimal currents}

We will denote by $I^k(M^n)$ the rectifiable integral $k$-currents on a riemannian manifold $M^n$.  For $I\in I^k(M^n)$ we denote by $|I|$ the associated varifold, where it is understood that if we apply $|I|$ to a subset of $M$ then we are evaluating $|I|$ on the pullback of this subset to the Grassmannian bundle.
In particular, the total mass of $I$ will be denoted $|I|(M^n)$.
By abuse of notation,
we will occasionally  refer to a general varifold on $M^n$ as $|I|$, even if {\it a priori} it does not arise from a current.  In practice, due to the context and results like Allard's regularity theorem \cite{Allard_firstvariation}, we will only be interested in varifolds which do arise in such a fashion. Therefore,
 the notational abuse is harmless.  If $U\subseteq M^n$ is a subset, then $I^k(M^n,U)$ will denote the integral $k$-currents $I$ such that $\partial I\subseteq U$.

We will be concerned with two classes of currents:  stationary currents and  minimizing rectifiable hypersurfaces.  The primary results hold in the varifold category as well.
By definition,  $I\in I^k(M^n,U)$ is minimizing means that if $I'\in I^k(M^n,U)$ satisfies $\partial I'=\partial I$, then $|I|(M)\leq |I'|(M)$.

 For stationary currents,  we will define  certain {\it quantitative stratifications} of the singular set.  Our main result,  Theorem \ref{t:current_quant_strat}, gives Minkowski content
estimates on the quantitative singular set,
which improve the more standard Hausdorff dimension estimates.  The primary application of this
new stratification, is to give new regularity results on minimizing hypersurfaces; see Theorem \ref{t:current_min_regularity}.
 These depend on the preliminaries given in  Section \ref{s:current_monotonicity}; compare Part I.

Recall that the basic assumptions  (\ref{con:sectional}), (\ref{con:volume}), on $M^n$
remain in force.

\subsection{The standard stratification}

Prior to defining the quantitative stratification of Theorem \ref{t:current_quant_strat}, we
will introduce some basic notions. These include the standard stratification of a stationary current,
which is based the notion of a "conical current".

\begin{definition}
 Let $J\in I^k(\dR^n)$ denote a rectifiable $k$-current on $\dR^n$.  We say that
$J$ is {\it $\ell$-conical at $y\in\dR^n$ with respect to the $\ell$-plane $V\subseteq\dR^n$}, if:
\begin{enumerate}
\item $J(r^*w)=r^k J(w)$ for all $k$-forms $w$, where $r^*w$ is the pullback of $w$ under the multiplication by $r$ map.
\item $J(w(y-z))=J(w(y))$ for all $k$-forms $w$ and $z\in V$.
\end{enumerate}
If $y=0^k$, we  say $J$ is $\ell$-conical.
\end{definition}
\begin{remark}
The first requirement is simply that $J$ is invariant under dilations. The second is that $J$ is invariant under translations by elements of $V$.  We will also say a varifold $|J|$ is {\it $\ell$-conical} if it satisfies these requirements.
\end{remark}

\noindent
Note that if $J\in I^k(\dR^n)$ is a $k$-conical current with respect to $V$,
 then $J$ is a multiple of $V$ as a current.

Next, we recall the notion of a "tangent cone" of a stationary current $I\in I^k(M)$.
In order to avoid certain cancelation issues,
it will be more convenient to think of the tangent cone of $I$ at $y$ as a varifold (although for our purposes
 the distinction is not really significant).

For $y\in I$ and $0<r<1$ we define the current $I_{y,r}$ on $B_{r^{-1}}(0^n)\subseteq \dR^n$ by:

$$
I_{y,r}(w):= r^{-k}I(\exp^*_y(r w))\, .
$$
We call $|I|_y$ a {\it tangent cone of $I$ at $y$} if there exists $r_i\rightarrow 0$ such that
$$
|I|_{y,r_i}\rightarrow |I|_y\, ,
$$
where convergence is in the weak$^*$ sense for varifolds.  Note that under the weak$^*$ topology,
 the space of varifolds on $\dR^n$ is a Frechet space with distance function $d_{*}(\cdot,\cdot)$.  Tangent cones at a point need not be unique but for a stationary current they are always $0$-conical.

Now,  we define the stratification,
$$
\cS^0(I)\subseteq \cS^1(I)\subseteq\cdots\subseteq \cS^{k-1}(I)=\cS(I)\subseteq M^n\, ,
$$
where by definition,
 $y\in \cS^\ell(I)$ if and only if no tangent cone $|I|_{y}$ at $y$  is $(\ell+1)$-conical.
If $I\in I^{n-1}(M)$ is area minimizing, then by classical results (see\cite{Almgren_InteriorRegMin}, \cite{Simon_MinVar}, \cite{Federer_GeomMeas})
\begin{equation}
\label{mde}
\dim\, \cS^\ell(I)\leq \ell\, ,
\end{equation}
\begin{equation}
\label{mcde}
\cS(I)=\cS^{n-8}(I)\, .
\end{equation}

\subsection{The quantitative stratification}

To define the quantitative stratification,  we need the notion of an "almost conical current".

\begin{definition}
\label{d:current_almost_conical}
An integral current $I\in I^k(M^n)$ is $(\eta,r,\ell)$-conical
 at $y\in I$ if there exists a $\ell$-conical varifold $|J|$ such that
$$
d_{*}(|I|_{y,r},|J|)<\eta\, .
$$
If $|J|$ is canonical with respect to the plane $V^\ell$,
then $|I|_{y,r}$ is said to be {\it $(\eta,r,\ell)$-canonical} with respect to $V^\ell$.
\end{definition}

Now we can introduce the quantitative singular set.

\begin{definition}
 For each $\eta>0$ and $0<r<1$, define the $\ell$-th effective singular stratum
$\cS^\ell_{\eta,r}(I)\subseteq M^n$ by
\begin{align}
 \cS^\ell_{\eta,r}(I)=\{y\in I: d_*(|I|_{y,s},|J|)>\eta\,\,{\rm for\,\, all}
\,\ell\text{-conical } |J|\,\,
{\rm and \,\,all}\, r\leq s\leq 1\}.
\end{align}
\end{definition}

The above definition can be rephrased as stipulating that $y\in \cS^\ell_{\eta,r}(I)$ if and only if
 $|I|_{y,s}$ is not $(\eta,s,\ell+1)$-conical for all $r\leq s\leq 1$.  Moreover, it follows immediately from the definition
that

$$
\label{e:current_containments}
\begin{aligned}
\cS^{\ell}_{\eta,r}(I)&\subset \cS^{\ell'}_{\eta',r'}(I)\qquad
({\rm if}\,\,\ell'\leq \ell,\,\,\eta'\leq\eta,\,\, r\leq r')\, ,\\
&{}\\
\cS^\ell(I)&=\bigcup_\eta \,\,\bigcap_r \, \cS^\ell_{\eta,r}(I)\, .
\end{aligned}
$$

Our main theorem on the effective singular set states that the known estimates
 on the Hausdorff dimension of
the singular set, given in \cite{Almgren_InteriorRegMin}, \cite{Simon_MinVar}, \cite{Federer_GeomMeas},
 can be strengthened to Minkowski content estimates on the quantitative singular set.  Equivalently, not only the effective singular sets themselves, but also tubular neighborhoods around these sets have controlled volume.

\begin{theorem}
\label{t:current_quant_strat}
Let the mass of the stationary integral current
 $I\in I^k(B_2(x),\partial B_2(x))$ satisfy
\begin{equation}
\label{mb}
|I|(B_2(x))<\Lambda\, .
\end{equation}
 Then for all $\eta>0$, there exists $C=C(n,\Lambda,\eta)$,
such that for all $0<r<1$,
\begin{equation}
\Vol(T_r(\cS^\ell_{\eta,r}(I))\cap B_1(x))\leq Cr^{n-\ell-\eta},
\end{equation}
\begin{equation}
|I|(T_r(\cS^\ell_{\eta,r}(I))\cap B_1(x))\leq Cr^{k-\ell-\eta}.
\end{equation}
\end{theorem}

\subsection{Quantitative estimates on the regularity scale}

To discuss the main consequences of Theorem \ref{t:current_min_regularity} we first define a notion of the regularity scale involving the second fundamental form $A$ of a current $I\in I^k(M)$.  This  is
an analog of the regularity
scale for maps introduced in Definition \ref{d:harm_reg_scale}.
 However, in the case of a currents, the quantitative nature of what we want to understand dictates that we
 allow for the possibility of multiplicity and multigraphs.  Hence we arrive at the following definition.

\begin{definition}
\label{d:current_reg_scale}
Let $r^N_{0,I}(x)=0$ if the current $I$ is not a union of at most $N$ $C^2$-submanifolds in some neighborhood of $x$.  Otherwise let $r^N_{0,I}(x)$ denote the maximum of
those $r>0$ such that $I\cap B_r(x)$ is such a union.   Define the {\it regularity scale} $r^N_I(x)$ by
\begin{align}
 r^N_I(x)= \max\left\{0\leq r\leq r^N_{0,I}(x): \sup_{B_r(x)}r|A|\leq 1\right\}\, .
\end{align}
\end{definition}
\noindent
Note that $r^N_I(x)$ is a scale invariant quantity. That is, if $r_f(x)=r$ and we rescale $B_r(x)$ to a unit size ball, then $I$ becomes a union of submanifolds with $|A|\leq 1$.  Moreover, if $I$ is stationary, then by standard elliptic regularity, if $r_I(x)\geq r$ then for all $k\in \dZ_+$, there exists $C_k$, depending
on the geometry of $M^n$, such that
$$
\sup_{B_{\frac{r}{2}}(x)}r^{k+1}|\nabla^{k} A|\leq C_k\, .
$$
Thus,  a bound on the regularity scale  $r^N_I(x)$ gives bounds for all derivatives of the second fundamental form in a definite sized neighborhood of $x$.

Next, based on the behavior the regularity scale, we partition $I$ into "good" and "bad" subsets.
\begin{definition}
Given a rectifiable current $I\in I^k(M)$ and any $r>0$ we define
\begin{align}
 \cB^N_{r}(I)= \{x\in I: r^N_{I}(x)\leq r\}.
\end{align}

\end{definition}

The following, Theorem \ref{t:current_min_regularity}, improves the Hausdorff dimension estimates on the singular set of a minimizing codimension 1 current in \cite{Almgren_InteriorRegMin}, \cite{Simon_MinVar}, \cite{Federer_GeomMeas}, to estimates on the regularity scale off sets of appropriately small volume.\footnote{In the case $n=8$ a stronger version was proved by Tom Ilmanen, who showed that away from
a {\it controlled} finite number of points $\{p_\alpha\}$ there is the bound
$|A|(x)\leq C \max |x-p_\alpha|^{-1}$. }

\begin{theorem}
\label{t:current_min_regularity}
Let $I\in I^{n-1}(B_2(x),\partial B_2(x))$ denote denote a minimizing current
with bounded mass
$$
|I|(B_2(x))<\Lambda\, .
$$
Then there exist $C=C(n,\Lambda, \eta)$, $N=N(n,\Lambda,\eta)$, such that
 for all $0<r<1$,
\begin{equation}
\label{smv1}
 \Vol(T_r(\cB^N_{r}(I))\cap B_{1}(x))\leq  Cr^{8-\eta}\, ,\\
\end{equation}
\begin{equation}
\label{smv2}
 |I|(T_r(\cB^N_{r}(I))\cap B_{1}(x))\leq  Cr^{7-\eta}.
\end{equation}
\end{theorem}

It follows immediately from Theorem \ref{t:current_min_regularity}  that
\begin{equation}
\label{Mindim}
\dim_{{\rm Min}}\cS(I)\leq n-8\, .
\end{equation}
The following stronger corollary
 is also an immediate consequence of Theorem  \ref{t:current_min_regularity},
together with the density estimate
$$
w_{n-1}r^{n-1} \leq |I|(B_r(y))\leq C(n)\Lambda r^{n-1}\, .
$$

\begin{corollary}
\label{c:current_min_regularity}
 For every $p<7$, there exists $C=C(n,\Lambda,p)$ such that if
$I\in I^{n-1}(B_2(x),\partial B_2(x))$ denotes a minimizing current
with mass $|I|(B_2(x))<\Lambda$, then
\begin{align}
\int_{I\cap B_1(x)}|A|^p\, d|I|\leq \int_{I\cap B_1(x)}r_{I}^{-p}\, d|I|\leq C.
\end{align}
\end{corollary}
\begin{remark}
By convention we take $|A|\equiv \infty$ on the singular set.
\end{remark}

As in Part I, we will focus on the $M^n=\dR^n$ case, as the more general case is identical up to some lower order corrections.  Additionally, since the techniques of this part of the paper mimic those of the harmonic maps section, when reasonable we will refer back to that section for details.

\section{Preliminary results}

In this section, we prove some preliminary results, which are required for our main theorems.  The
quantitative rigidity
theorem, Theorem \ref{t:current_monotone_rigidity}, is the counterpart of the ``almost volume cone implies almost metric cone" theorem of \cite{cheegercoldingalmostrigidity} for manifolds with lower Ricci
curvature bounds; compare also Theorem \ref{t:harm_monotone_rigidity} for harmonic maps.
In essence, it states that if the density of a stationary current (defined in (\ref{e:current_density}))
does not change
much between two scales, then the current is almost conical in the sense of Definition \ref{d:current_almost_conical}.  As in \cite{cheegernaber_quantstrateistein}, this plays an important role in the proof of quantitative stratification theorem, Theorem \ref{t:current_quant_strat}.  Theorem \ref{t:current_eps_regularity} is an $\epsilon$-regularity theorem which, when combined with Theorem \ref{t:harm_quant_strat}, gives Theorem \ref{t:current_min_regularity}, our main result for the regularity scale. This also parallels the discussion of \cite{cheegernaber_quantstrateistein},
as well as that of Part I.

\subsection{Monotonicity, quantitative rigidity and $\epsilon$-regularity}\label{s:current_monotonicity}

In this section we consider currents $I\in I^k(\dR^n)$.
Given  $I$, we define for $x\in \dR^n$ and $0<r<1$ the density, $\theta_r(x)$, by

\begin{align}\label{e:current_density}
\theta_r(x)=r^{-k}|I|(B_r(x)).
\end{align}

It is well known that if $I$ is a stationary current, or varifold, then for each $x$ this function is monotone nondecreasing in $r$ with $\theta_s(x)=\theta_r(x)$ if and only if $I$ is dialation invariant on the annulus $A_{s,r}(x)$;  see \cite{Federer_GeomMeas}.  More precisely, if $(x-y)^N$ denotes the projection of $(x-y)$ to the perpendicular of $I$ (which makes sense $a.e.$) then
 we have the following monotonicity formula
\begin{align}
\theta_r(x)-\theta_s(x) = \int_{A_{s,r}(x)\cap I} \frac{|(x-y)^N|^2}{|x-y|^{k+2}}\, ,
\end{align}
and in particular, $\theta_r(x)$ is monotone nondecreasing:
$$
\theta_r(x)\uparrow\, .
$$
  The following quantitative rigidity theorem is then an immediate consequence of this point and a contradiction argument, see Theorem \ref{t:harm_monotone_rigidity}.

\begin{theorem}[Quantitative rigidity]
\label{t:current_monotone_rigidity}
For all $ \Lambda,\epsilon >0$, there exists $\delta=\delta(n,\Lambda,\epsilon)>0$,
$r=r(n,\Lambda,\epsilon)>0$, such that if $I\in I^k(B_2(0^n),\partial B_2(0^n))$ denotes a stationary current with $|I|(B_2(0))<\Lambda$,
satisfying
\begin{align}\label{e:current_almostrigid}
\theta_1(0^n)-\theta_r(0^n)\leq \delta\, ,
\end{align}
then $I$ is $(\epsilon,1,0)$-conical at $0^n$.
\end{theorem}

Next, we focus on {\it minimizing} currents $I\in I^{n-1}(B_2(0^n),\partial B_{2}(0^n))$.  In this context, we prove the necessary $\epsilon$-regularity theorem which enables us to turn Theorem \ref{t:current_quant_strat} into an effective estimate on the regularity scale.

Recall that if $I\in I^{n-1}(B_2(0^n),\partial B_{2}(0^n))$, then by the homological theorem 4.5.17 of \cite{Federer_GeomMeas}, there exist at most countably many open subsets $A_i\subseteq B_2(0)$ with boundaries $\partial A_i\in I^{n-1}(B_2(0^n))$ such that
\begin{equation}
\label{countably}
I=\sum_1^\infty\partial A_i\cap B_2(0^n)\, ,
\end{equation}
\begin{equation}
\label{countably1}
|I|=\sum |\partial A_i|\, .
\end{equation}
In particular, if $I$ is minimizing, then so are the $\partial A_i$.  Moreover, if $I$ satisfies the mass bound
$|I|(B_2(0^n))<\Lambda$, then by a density argument, at most $N(n,\Lambda)$ of the $\partial A_i$'s have support in $B_{\frac{3}{2}}(0^n)$.  We will call a minimal current $I\in I^{n-1}(B_2(0^n),\partial B_{2}(0^n))$ {\it indecomposable} if $I=\partial A$ for some such $A$.

In the proof of the following we will assume Theorem \ref{t:current_quant_strat}.  The remainder of the paper will be devoted to Theorem \ref{t:current_quant_strat}, and its proof is quite independent of the following.

\begin{theorem}[$\epsilon$-regularity]\label{t:current_eps_regularity}
 For all $\Lambda>0$,  there exists $\epsilon=\epsilon(n,\Lambda)>0$ with the following property.  Let $I\in I^{n-1}(B_{2}(0^n),\partial B_{2}(0^n))$ satisfy $|I|(B_{2}(0^n))<\Lambda$ and assume in addition
that I is indecomposable, minimizing, and  $(\epsilon,2,n-7)$-conical. Then $r^1_I(x)> 1$.
\end{theorem}
\begin{proof}[Proof of Theorem \ref{t:current_eps_regularity} given Theorem \ref{t:current_quant_strat}]
It is a consequence of \cite{Almgren_InteriorRegMin} that there exists $\epsilon(n,\Lambda)>0$ such that if $I$ is $(\epsilon,2,n-1)$-conical then $r^1_I(x)>1$.  It  follows that if $\eta\leq \eta(n,\Lambda)$, then $\cB_{r}(I)\subseteq \cS^{n-2}_{\eta,r}(I)$.  In particular, for all sufficiently small $\eta$, we obtain
$$
\Vol(T_r(\cB_r(I))\cap B_1(x))\leq
\Vol(T_r(\cS^{n-2}_{\eta,r}(I)))\leq C(n,\Lambda,\eta)r^{2-\eta}\, ,
$$
and
\begin{equation}\label{weak}
|I|(T_r(\cB_r(I))\cap B_1(x))\leq \Vol(T_r(\cS^{n-2}_{\eta,r}(I)))\leq C(n,\Lambda,\eta)r^{1-\eta}\, .
\end{equation}

This estimate, which holds for any $I$ which is indecomposable and minimizing, plays the role of a weak version of Theorem \ref{t:current_min_regularity}.

Relation (\ref{weak}) implies the following:  If $I_i\in I^{n-1}(B_2(0^n),\partial B_2(0^n))$ is a sequence of indecomposable minimizing currents with $|I_i|(B_2(0^n))<\Lambda$, then a subsequence $I_j$
converges weakly to a limit indecomposable minimizing current
$
I_\infty\in I^{n-1}(B_2(0^n),\partial B_2(0^n))\, ,
$
and in addition,
$$
|I_i|\rightarrow |I_\infty|\, .
$$
This holds because the convergence is of smooth single valued graphs away from a set of $(n-1)$-measure zero.  In particular the limit of the varifolds $|I_i|$ is the varifold $|I_\infty|$ and there is no cancelation  of the rectifiable currents.

Now we can finish the proof of the Theorem.
Assume there does not exist $\epsilon(n,\Lambda)$
as in the statement of the theorem.
Then we can find a sequence $\epsilon_i\rightarrow 0$ with indecomposable,
minimizing $I_i\in I^{n-1}(B_{2}(0^n),\partial B_{2}(0^n))$ such that
$|I_i|(B_{2}(0^n))<\Lambda$ and $I_i$ is $(\epsilon_i,2,n-7)$-conical, but $r^1_{I_i}(0^n)\leq 1$.
After passing to a subsequence we can let
$I_{i}\rightarrow I_\infty\in I^{n-1}(B_2(0^n),\partial B_2(0^n))$ as above.
Now $I_\infty$ is a minimizing current which is $(n-7)$-conical. Thus, it follows from
\cite{Simon_MinVar} that $I_\infty$ is a hyperplane.  Further,
 since $|I_{i}|\rightarrow |I_\infty|$ as in the last paragraph,
we have for $\delta_i\rightarrow 0$ that $I_{i}$ is
 $(\delta_i,2,n-1)$-conical.  Hence, for $i$ sufficiently large, the proof can completed by
 appealing to  the statements established in the preceding   paragraphs.
\end{proof}

\subsection{Reduction of Theorem \ref{t:current_min_regularity} to Theorem \ref{t:current_quant_strat}}\label{s:current_Reg2Stratreduction}

We now show that Theorem \ref{t:current_min_regularity} follows
easily from  Theorems \ref{t:current_quant_strat} and \ref{t:current_eps_regularity}.

\begin{proof}[Proof of Theorem \ref{t:current_min_regularity}]

As in (\ref{countably}), (\ref{countably1}), write  $I\in I^{n-1}(B_2(0^n),\partial B_2(0^n))$
as a union
$$
I=\sum_1^{N(n,\Lambda)}\partial A_i\, ,
$$
such that $|I|=\sum |\partial A_i|$.  Thus, by applying the arguments to each piece,
$\partial A_i$, we can assume without loss of generality, that $I$ is indecomposable.

It follows from  Theorem \ref{t:current_eps_regularity}
 that if $\eta\leq \eta(n,\Lambda)$ then
$\cB_{r}(I)\subseteq \cS^{n-8}_{\eta,r}(I)$.
  In particular, for all suffiently small
$\eta$,  we have $T_r(\cB_r(f))
\subseteq T_r(\cS^{n-8}_{\eta,r}(I))$.
Therefore, we obtain (\ref{smv1}), (\ref{smv2}),
$$
\Vol(T_r(\cB_r(I))\cap B_1(x))\leq
\Vol(T_r(\cS^{n-8}_{\eta,r}(I)))\leq C(n,\eta,\Lambda)r^{8-\eta}\, ,
$$
$$
|I|(T_r(\cB_r(I))\cap B_1(x))\leq \Vol(T_r(\cS^{n-8}_{\eta,r}(I)))\leq C(n,\eta,\Lambda)r^{7-\eta}\, ,
$$
which completes the proof of Theorem \ref{t:current_min_regularity}.
\end{proof}

\section{The decomposition lemma and proof of Theorem \ref{t:current_quant_strat}}\label{s:current_decomposition}
Modulo  obvious changes from harmonic maps to stationary currents,
the arguments of this section
parallel almost verbatim those of Sections
 \ref{s:harm_decomposition}, \ref{s:harm_cone_splitting}.
Therefore,
 we will refer to these sections for most details.

 We begin by describing the Decomposition Lemma for stationary currents.
As with harmonic maps, the goal is cover the effective singular sets
$$
\cS^\ell_{\eta,\gamma^j}(I)\subseteq\cup \cC^\ell_{\eta,\gamma^j}
$$
by a not too large collection of sets $\{\cC^\ell_{\eta,\gamma^j}\}$,
each of which is itself a union of a controlled number of balls
$\cC^\ell_{\eta,\gamma^j} =\cup B_{\gamma^j}$.  Given appropriate
quantitative
control on the number of such sets $\{\cC^\ell_{\eta,\gamma^j}\}$ and the number of balls $B_{\gamma^j}$ in each $\cC^\ell_{\eta,\gamma^j}$ (see Lemma \ref{l:current_decomposition})
 the proof of Theorem \ref{t:current_quant_strat} follows almost
 immediately.  Roughly speaking,
 the sets $\{\cC^\ell_{\eta,\gamma^j}\}$ are constructed by grouping
together those points of $\cS^\ell_{\eta,\gamma^j}(I)$ whose conical
behavior on the same scales is similar; see Sections
\ref{s:harm_decomposition} and \ref{s:harm_cone_splitting}
 for a more complete description.

\begin{lemma}[Decomposition Lemma]\label{l:current_decomposition}
There exists $c_1(n), c_0(n),K(n,\Lambda,\eta,\gamma), Q(n,\Lambda,\eta,\gamma)>0$ such that for each $j\in\dZ_+$:
\begin{enumerate}
\item The set $\cS^\ell_{\eta,\gamma^j}(I)\cap B_1(0)$ is contained in the union of at most
$j^K$ nonempty sets $\cC^\ell_{\eta,\gamma^j}$.
\item Each set $\cC^\ell_{\eta,\gamma^j}$  is the union of at most
$(c_1\gamma^{-k})^{Q}\cdot (c_0\gamma^{-\ell})^{j-Q}$ balls of radius $\gamma^j$.
\end{enumerate}
\end{lemma}

Given the Decomposition Lemma we observe that the proof of Theorem \ref{t:current_quant_strat} now follows as in Section \ref{s:harm_decomposition}.  Note however, that for the second estimate:
$$
|I|(T_r(\cS^\ell_{\eta,r}(I))\cap B_1(x))\leq Cr^{k-\ell-\eta}\, ,
$$
one should replace the volume estimate (\ref{e:harm_volume}) with the estimate
$$
w_k r^k\leq |I|(B_r(x))\leq c(n,\Lambda) r^k\, .
$$
This estimate itself follows immediately from the monotonicity formula.

The proof of Lemma \ref{l:current_decomposition} itself follows directly from a Covering Lemma whose statement is
analogous to that of Lemma \ref{l:harm_splitting}.  As in
Section \ref{s:harm_cone_splitting}, the proof
of the Covering Lemma follows from a Cone-splitting Lemma
which is stated in the next subsection.

\subsection{Cone-splitting Lemma}
\label{s:current_cone_splitting}

We begin with the following observation; compare Section \ref{s:harm_cone_splitting}.
\vskip2mm
\noindent
{\bf Cone-splitting principle for varifolds.}
Let $|I|$ denote a $k$-varifold on $\dR^n$ which is $\ell$-conical with respect to the
$\ell$-plane $V^\ell$.  Assume in addition that for some $y\not\in V^\ell$ that $|I|$ is also
$0$-conical with respect to $y$.  Then it follows that $|I|$ is $(\ell+1)$-conical
with respect to the $(\ell+1)$-plane ${\rm span}\{y,V^\ell\}$.
\vskip1mm

From this observation and a contradiction argument as in Theorems \ref{t:current_monotone_rigidity}, \ref{t:harm_monotone_rigidity} we immediately
obtain the following.

\begin{lemma}[Cone-Splitting Lemma]
\label{l:current_cone_splitting}
 For all $\Lambda, \epsilon,\tau,\gamma>0$ there exists
$\delta=\delta(n,\Lambda,\epsilon,\tau,\gamma)>0$, such that the following holds: Let
$I\in I^k(B_{\gamma^{-1}}(0^n),\partial B_{\gamma^{-1}}(0^n))$
denote a stationary current with mass $|I|(B_{\gamma^{-1}}(0^n))<\Lambda$,
 such that:
\begin{enumerate}
 \item $I$ is $(\delta,\gamma^{-1},\ell)$-conical at $0$ with respect to $V^\ell_0$.
 \item There exists $y\in B_1(0^n)\setminus T_{\tau}(V^\ell_{0})$ such that $I$ is
$(\delta,2,0)$-conical at $y$.
\end{enumerate}
Then $I$ is $(\epsilon,1,\ell+1)$-conical at $0^n$.
\end{lemma}

As in Lemma \ref{l:harm_cone_splitting}, the import of the present
Cone-splitting Lemma is  that when {\it almost conical} points of $I$ are
close to one another,
 then they interact and force a surrounding neighborhood of $I$
to have a larger symmetry group.  As in Section
\ref{s:harm_cone_splitting}, we can use  induction to obtain the following
corollary.  This corollary is what is  needed for the proof of the relevant
Covering Lemma.

\begin{corollary}
\label{c:current_inductive}
 For all $\epsilon,\tau,\gamma >0$ there exists
$\delta=\delta(n,\Lambda, \epsilon,\tau,\gamma)>0$,
$\theta=\theta(n,\Lambda,\epsilon,\tau,\gamma)>0$,
 such that the following holds.
Let $I\in I^k(B_{2}(0^n),\partial B_2(0^n))$ denote
 a stationary current with mass $|I|(B_{2}(0^n))|<\Lambda$.
Let $r\leq \theta$ and $x\in L_{\gamma^{-n},r,\delta}(I)\cap B_1(0^n)$.
Then there exists $0\leq \ell\leq n$ such that:
\begin{enumerate}
\item $I$ is $(\epsilon,r,\ell)$-conical at $x$ with respect to $V^\ell_{x}$.
\item $L_{\gamma^{-n},\delta,r}\cap B_r(x)\subseteq T_{\tau r}(V^\ell_{x})\, .$
\end{enumerate}
\end{corollary}

\bibliography{rgcn}
\bibliographystyle{alpha}

\end{document}